\documentclass[12pt,twoside]{article}
\usepackage{graphicx}
\usepackage{amsfonts,amsbsy,amssymb,amsmath}
\allowdisplaybreaks
\usepackage{cite}
\usepackage{float}
\usepackage{caption}
\usepackage{graphics}
\usepackage{xcolor}
\usepackage{hyperref}
\usepackage{subfigure}
\usepackage{enumerate}
\def\bpsp{\begin{pspicture}}
\def\epsp{\end{pspicture}}

\newtheorem{theorem}{Theorem}[section]
\newtheorem{remark}[theorem]{Remark}
\newtheorem{example}[theorem]{Example}
\newtheorem{lemma}[theorem]{Lemma}
\newtheorem{corollary}[theorem]{Corollary}
\newtheorem{definition}[theorem]{Definition}
\newtheorem{proposition}[theorem]{Proposition}

\newtheorem{note}{Note}
\newtheorem{case}{Case}

\newtheorem{conjecture}{Conjecture}
\newtheorem{question}{Question}

\newcommand{\bea}{\begin{eqnarray}}
\newcommand{\eea}{\end{eqnarray}}
\newcommand{\beq}{\begin{eqnarray*}}
\newcommand{\eeq}{\end{eqnarray*}}

\def\m4{\mbox{\rm ~(mod $4$)}}

\def \bd{\begin{definition}}
\def \ed{\end{definition}}
\def \bqu{\begin{question}}
\def \equ{\end{question}}
\def \bcc{\begin{conjecture}}
\def \ecc{\end{conjecture}}
\def \bt{\begin{theorem}}
\def \et{\end{theorem}}
\def \bl{\begin{lemma}}
\def \el{\end{lemma}}
\def \bc{\begin{corollary}}
\def \ec{\end{corollary}}
\def \be{\begin{equation}}
\def \ee{\end{equation}}
\def \ben{\begin{enumerate}}
\def \een{\end{enumerate}}
\def \ba{\begin{array}}
\def \ea{\end{array}}
\def \bp{\begin{proposition}}
\def \ep{\end{proposition}}
\def \bx{\begin{example}}
\def \ex{\end{example}}
\def \br{\begin{remark}}
\def \er{\end{remark}}
\def \bdsc{\begin{description}}
\def \edsc{\end{description}}

\def \bn{\begin{case}}
\def \en{\end{case}}
\def \bnt{\begin{note}}
\def \ent{\end{note}}
\def\1{1\!\!1}

\def\mm2{\mbox{\rm ~(mod $2$)}}
\def\m4{\mbox{\rm ~(mod $4$)}}

\def\qed{\nolinebreak\hfill\rule{.2cm}{.2cm}\par\addvspace{.5cm}}

\def\m{\mu}

\def\1{\textbf{1}}
\def\0{\textbf{0}}

\begin{document}
\title{ The spectral extrema of graphs of odd size forbidding $H(4,3)$ beyond the book graph}
\author{Abdul Basit Wani$ ^{a} $, S. Pirzada$ ^{b} $, Amir Rehman $ ^{c} $\\
$^{a,b}${\em Department of Mathematics, University of Kashmir, Srinagar, India}\\
$^{c}${\em Department of Mathematics, Govt. Degree College, Frisal, J\&K, India}\\
$^{a}$\texttt{waniiibasit@gmail.com }; $ ^{b} $\texttt{pirzadasd@kashmiruniversity.ac.in}\\
$^{c}$\texttt{aamirnajar786@gmail.com }
}

\date{}

\pagestyle{myheadings} \markboth{Wani, Pirzada, Rehman}{The spectral extrema of graphs of odd size forbidding $H(4,3)$ beyond the book graph}
\maketitle

\noindent{\footnotesize \bf Abstract.}
A graph is said to be $H$-free if it does not contain a subgraph isomorphic to $H$. The fish graph, denoted by $H(4, 3)$, is a $6-$vertex graph obtained from a cycle of length $4$ and a triangle by sharing a common vertex. Earlier it is shown that $\lambda(G)\leq \frac{1+\sqrt{4m-3}}{2}$ holds for all $H(4,3)-$free graphs of odd size $m\geq 44,$ and the equality holds if and only if $G\cong S_{\frac{m+3}{2},2},$ where $S_{\frac{m+3}{2},2}$ is the $m-$edge book graph $K_2 \vee \frac{m-1}{2}K_1,$ where $K_2 \vee \frac{m-1}{2}K_1,$ denotes the join of $K_2$ and $\frac{m-1}{2}K_1.$ Let $\mathcal{G}(m,H(4,3))$ denote the family of $H(4,3)$-free graphs with $m$ edges and no isolated vertices. We write
\(
\mathcal{G}(m,H(4,3)) \setminus \left\{ K_2 \vee \tfrac{m-1}{2}K_1 \right\}
\)
for the corresponding subfamily obtained by excluding the book graph.
 In this paper, we establish a sharp upper bound on the spectral radius of graphs over $\mathcal{G}(m,H(4,3))\setminus \{K_2 \vee \frac{m-1}{2}K_1\}$ for odd $m\geq 58$ and characterize the unique extremal graph attaining this bound.\\

\vskip 3mm



\section{Introduction}
\indent Let $G(V,E)$ be a simple graph of order $|V(G)|=n$ and size $|E(G)|=m$. Let $A(G)$ be the adjacency matrix of $G$ and $\lambda(G)$ be the spectral radius of $G$. Since $A(G)$ is real symmetric, its eigenvalues are real. By the Perron--Frobenius theorem, if $G$ is connected, then there exists a positive unit eigenvector $\mathbf{x} = (x_1, x_2, \ldots,x_n)^T$ corresponding to $\lambda(G)$, which is called the Perron vector of $G$. Obviously, $0 < x_i < 1$ for  $1 \le i \le n.$ It will be convenient to associate weights to the vertices of $G$ (with respect to $\mathbf{x}$), where $x_v$ is the weight of vertex $v$. Also, the quadratic form $\mathbf{x}^T A(G) \mathbf{x}$ can be written as
\begin{equation*}
\lambda(G) = \mathbf{x}^T A(G) \mathbf{x} = \sum_{uv \in E(G)} 2x_u x_v.
\end{equation*}
The eigenvector equation $A(G)\mathbf{x} = \lambda(G)\mathbf{x}$ can be interpreted as
\begin{equation}\label{111}
(A(G)\mathbf{x})_v = \lambda x_v = \sum_{u \in N_G(v)} x_u,
\end{equation}
for each $v \in V(G)$.\\
\indent Let $\mathcal{F}$ be a set of graphs. We say that a graph $G$ is $\mathcal{F}$-free if it does not contain any member of $\mathcal{F}$ as a subgraph, and $\mathcal{F}$ is called the forbidden set. When $\mathcal{F}$ is a singleton, say $\mathcal{F}=\{F\}$, we simply say that $G$ is $F$-free. Let $\mathcal{G}(m,\mathcal{F})$ denote the set of all $\mathcal{F}$-free graphs with $m$ edges and no isolated vertices.\\
\indent For a subset $S\subset V(G)$, let $G[S]$  denote the subgraph of $G$ induced by $S$. We denote by $e(S,T)$ the number of edges with one end in $S$ and the other end in $T$, where $S$ and $T$ are subsets of $V(G)$. We write $e(S,S)$ by $e(S)$. Let $N_k(v)$ be the set of vertices of distance $k$ from $v$. In particular, $N_1(v)=N(v)$. Let $N[v]=N(v)\cup\{v\}$. For $S \subset V$, let $N_S(v)$ be the set of neighbours of $v$ in $S$ and let $d_S(v)=|N_S(v)|$. If no ambiguity is possible, we put $N(v) = N_S(v)$ and $d(v) = d_S(v).$\\
\indent Given two graphs $G$ and $H$, the graph $G\vee H$ denotes the join and $G\cup H$ denotes the union. Let $kG$ denote the union of $k$ vertex-disjoint copies of $G$. As usual, we use $P_n$, $C_n$, $K_n$, $K_{t,n-t}$ to denote the path, the cycle, the complete graph, the complete bipartite graph of order $n$, respectively. In particular, $K_{1,n-1}$ is also known as the $n$-vertex star. Let $S_{n,2}$ be the graph obtained by joining each vertex of $K_2$ to $n-2$ isolated vertices. Let $H(\ell,3)$ denote the graph obtained from a cycle $C_\ell$ and a triangle by sharing a common vertex. For more definitions and notations, we refer to \cite{SP}.\\
 \indent A classical problem in the extremal graph theory is the Tur\'an problem which asks what is the maximum size of an $H$-free graph of order $n$, where the maximum size is known as the Tur\'an number of $H$. Nikiforov \cite{NF3}, proposed a spectral analogue of the Tur\'an problem which asks what is the maximum spectral radius of an $H$-free graph of size $m$ or order $n$. In \cite{EN}, Nosal proved that $\lambda(G) \le \sqrt{m}$ for every graph of size $m$ when $H$ is a triangle. Lin, Ning and Wu \cite{LNW} improved the bound to $\lambda(G) \le \sqrt{m-1}$, when $G$ is non-bipartite and triangle-free. Nikiforov \cite{NF2} showed that $\lambda(G) \le \sqrt{m}$ for all $C_4$-free graphs of size $m$. The problem has been investigated for various graphs $H$, such as \cite{CLZ,CDT,NF1,NF2,SpAr1, AP,SpAr,ArSp,ZHAI}.\\
\indent In 2023, Li, Lu, and Peng~\cite{LLP} investigated the spectral extremal problem for $F_k$-free graphs with a fixed number of edges. They showed that if $G$ is an $F_2$-free graph with $m$ edges, then
\(
\lambda(G) \le \frac{1+\sqrt{4m-3}}{2},
\)
and equality holds if and only if $G$ is the join of $K_2$ and an independent set of $\tfrac{m-1}{2}$ vertices for $m \ge 8$. It is worth noting that Li et al.\ used the notation $F_2$ to denote the graph $H(3,3)$. However, the above bound cannot be achieved when $m$ is even. This case was subsequently studied by Chen et al.~\cite{L21} and they also derived a sharp upper bound on the spectral radius for graphs in $\mathcal{G}(m,F_2)\setminus \left\{ K_2 \vee \frac{m-1}{2}K_1 \right\}$.
In \cite{ArSp}, it is shown that if $G$ is an  $\{H(3,3), H(4,3)\}$-free graph of odd size, then $\lambda(G)\le\frac{1+\sqrt{4m-3}}{2}$, with equality if and only if $G\cong S_{\frac{m+3}{2},2}$. Zhang and Wang~\cite{L23} and independently by Rehman and Pirzada \cite{RP} characterized the unique extremal graph with maximum spectral radius over all graphs in $\mathcal{G}(m,H(4,3))$ under certain conditions.\\
\indent Motivated by the  aforementioned works, we further investigate the extremal problem for the spectral radius by excluding the extremal graph
$K_2 \vee \tfrac{m-1}{2}K_1$ identified in~\cite{RP,L23} from the family
$\mathcal{G}(m,H(4,3))$. Our main result is stated in the following theorem.

\vskip 3mm
\begin{theorem}\label{MT}
Let $G \in \mathcal{G}(m,H(4,3))\setminus\left\{K_2 \vee \frac{m-1}{2}K_1\right\}$ with odd size $m \ge 58$.
Then $\lambda(G) \le \tilde{\lambda}(m),$
with equality if and only if $G \cong K_1 \vee \left( K_{1,\frac{m-3}{2}} \cup 2K_1 \right),$ where $\tilde{\lambda}(m)$ is the largest root of $x^4 - m x^2 - (m-3)x + m - 3 = 0.$
\end{theorem}

\section{\noindent \textbf{Preliminaries}}
In this section, we present some necessary preliminary results, which will be used to prove our main results.\\
\indent Let $M$ be a real symmetric matrix whose rows and columns are indexed by
$V=\{1,2,\ldots,n\}$. Assume that $M$ can be written as

\[M=
\begin{pmatrix}
M_{11} & \cdots & M_{1s} \\
\vdots & \ddots & \vdots \\
M_{s1} & \cdots & M_{ss}
\end{pmatrix}
\]\\
according to the partition $\pi: V = V_1 \cup V_2 \cup \cdots \cup V_s,$
where $M_{ij}$ denotes the submatrix (block) of $M$ formed by the rows in $V_i$ and the columns in $V_j$. Let $q_{ij}$ denote the average row sum of the block $M_{ij}$. Then the matrix $M_{\pi} = (q_{ij})$
is called the \emph{quotient matrix} of $M$ with respect to the partition $\pi$. If the row sum of each block $M_{ij}$ is constant, then the partition $\pi$ is called an \emph{equitable partition}.\\
 A well known observation in the literature of spectral graph theory is as follows.
\begin{lemma}{\em \cite{L26}} \label{lem28}
    Let $ M$ be a real matrix with an equitable partition $\pi$, and let $M_\pi$ be the corresponding quotient matrix. Then every eigenvalue of $M_\pi$ is an eigenvalue of $M$. Furthermore, if $M$ is nonnegative and irreducible, the largest eigenvalues of $M$ and $M_\pi$ are equal.
\end{lemma}
\begin{lemma}{\em \cite{L21}} \label{lem20}
    If $m \ge 6$, then $\lambda\left(K_{1}\vee\bigl(K_{1,\frac{m-3}{2}}\cup 2K_{1}\bigr)\right)$is the largest root of $x^{4}-m x^{2}-(m-3)x+m-3=0,$
and $\lambda\!\left(K_{1}\vee\bigl(K_{1,\frac{m-3}{2}}\cup 2K_{1}\bigr)\right)> \frac{1+\sqrt{4m-7}}{2}.$
\end{lemma}
\indent The next lemma gives the local structure of  $ H(4, 3)-$ free graphs.
\begin{lemma}{\em \cite{L23}} \label{lem24}
 Let $G$ be an $ H(4, 3)-$ free graph and $u$ be any vertex of $V(G).$ If $G[N_G(u)]$
contains at least $7$ edges, then each non-trivial component of $G[N_G(u)]$ is isomorphic
to a star $ K_{1,r} $ for some $r\geq 1.$
\end{lemma}
\indent The following lemma  plays an important role in comparing the spectral radii of graphs obtained via edge relocation.
\begin{lemma}{\em \cite{L24}} \label{lem25}
  Let $G$ be a connected graph and $ X=(x_1,x_2,\dots,x_n)^T $ be its Perron vector, where $x_i$ is corresponding to $ v_i(1\leq i\leq n)$. Assume that $u,v$ are two distinct vertices in $V(G)$ with $\{v_i|i=1,2,\dots,s\}\subseteq N_G(v)\setminus N_G(u)$. Let $ G^*=G-\{vv_i|1\leq i\leq s\}+\{uv_i|1\leq i\leq s\}$. If $x_u\geq x_v,$ then $ \lambda(G)< \lambda(G^*)$.
\end{lemma}

\begin{lemma}{\em \cite{L25}} \label{lem26}
  If $\lambda(G) > \frac{1+\sqrt{4m-5}}{2} $ and there exists a vertex $v_{j}$ of $G$ such that $x_{j} < (1-\alpha)x_{u^*}$, where $0<\alpha<1,$ then
  \begin{align*}
  e(W)< e(U) - |U_+| + \frac{3}{2}- \alpha\, d_{U}(v_{j}),\quad \text{for } v_{j}\in W,
  \end{align*}
and
 \begin{align*}
 e(W)< e(U) - |U_+| + \frac{3}{2}- \alpha\bigl(d_{U}(v_{j})-1\bigr),\quad \text{for } v_{j}\in U_+.
   \end{align*}
\end{lemma}

\begin{lemma}{\em \cite{L21}} \label{lem27}
    If $\lambda(G) >\frac{ 1+\sqrt{4m-7}}{2}$ and there exists a vertex $v_{j}$ of $G$ such that $x_{j} < (1-\alpha)x_{u^*},$ where  $0<\alpha<1,$ then
 \begin{align*}
 e(W)< e(U) - |U_+| + 2- \alpha\, d_{U}(v_{j}),\quad \text{for } v_{j}\in W,
   \end{align*}
and
 \begin{align*}
 e(W)< e(U) - |U_+| + 2- \alpha\bigl(d_{U}(v_{j})-1\bigr),\quad \text{for } v_{j}\in U_+.
  \end{align*}
\end{lemma}

\section{\noindent \textbf{Proof of the main theorem}}

In view of Lemma \ref{lem20}, it suffices to prove that $\lambda(G) \le \frac{1 + \sqrt{4m - 7}}{2}$ when $G \in \mathcal{G}(m,H(4,3))\setminus\left\{ K_2 \vee \frac{m-1}{2}K_1,\;
K_1 \vee \left( K_{1,\frac{m-3}{2}} \cup 2K_1 \right) \right\}.$\\
\indent Choose $\hat{G} \in \mathcal{G}(m,H(4,3))\setminus\left\{ K_2 \vee \frac{m-1}{2}K_1 \right\}$ with odd size $m \ge 58$ such that the spectral radius of $\hat{G}$ is as large as possible. We denote $\lambda := \lambda(\hat{G})$ for short. Assume that $x = (x_1,x_2,\ldots,x_n)^T$ is the Perron vector of $\hat{G}$ with $x_{u^*} = \max\{x_u \mid u \in V(\hat{G})\}.$ Define $U=N_{\hat G} (u^*)$ and $ W = V(\hat G) \setminus(U \cup \{u^*\})$. Moreover we partition $U$ into $U_0$ and $U_+$, where $U_0=\{ u\in U | d_U(u)=0\}$ and $U_+=U\setminus U_0.$ It is easy to see that
\begin{equation}\label{310}
     m=|U|+e(U_+)+e(U,W)+e(W)
\end{equation}
Also, we have
\begin{equation}\label{311}
    \lambda x_{u^*}=\sum_{u\in U}x_u=\sum_{u\in U_+}x_u + \sum_{v\in U_0}x_v,
    \end{equation}
    and
    \begin{equation}\label{312}
        \lambda^2 x_{u^*}=|U|x_{u^\ast}+\sum_{u\in U_+}d_U(u)x_u + \sum_{w\in W}d_U(w)x_w.
    \end{equation}
    Moreover, from  (\ref{311}) and (\ref{312}), we obtain
\begin {align}\label{E3}
({\lambda}^2-\lambda) x_{u^\ast}=|U| x_{u^\ast}+\sum_{v\in U_+}(d_U(v)-1)x_v+\sum_{w\in W}d_U(w)x_w-\sum_{v\in U_0}x_v.
\end{align}

On the contrary, we suppose that $\lambda >\frac{1 + \sqrt{4m - 7}}{2}$
and  $\hat{G} \not\cong K_1 \vee \left( K_{1,\frac{m-3}{2}} \cup 2K_1 \right).$ Then, \({\lambda}^2 - \lambda \ge m-2\). By substituting \eqref{310}
in (\ref{E3}), we obtain
\begin{equation}\label{555}
e(W) < e(U_+) - |U_+| - \sum_{v \in U_0} \frac{x_v}{x_{u^\ast}} + 2.
\end{equation}

We first establish a sequence of lemmas, each of which plays a crucial role in characterizing the structure of the extremal graph.

\begin{lemma}\label{C1}
$\hat{G}[U_{+}]$ contains at least $7$ edges.
\end{lemma}
\noindent \textbf{Proof.} Assume, to the contrary, that $e(U_{+}) \le 6$.
Then, by \eqref{310} and \eqref{312}, we obtain
\begin{align*}
\lambda^{2} x_{u^{*}}
  &= |U|x_{u^{*}}+ \sum_{u \in U_{+}} d_{U}(u)x_{u}+ \sum_{w \in W} d_{U}(w)x_{w}\\
  &\leq \bigl(|U| + 2e(U_{+}) + e(U,W)\bigr)x_{u^{*}}\\
  &= \bigl(m - e(W) + e(U_{+})\bigr)x_{u^{*}}\\
  &\leq (m + 6)x_{u^{*}}.
\end{align*}
Therefore, we can find that $\lambda \le \sqrt{m + 6} \leq \frac{1 +{\sqrt{4m-7}}}{2}$
 for any $m \ge 58,$ a contradiction. \qed

\begin{lemma}\label{C2}
$\hat{G}[U_+]$ is isomorphic to a star $K_{1,r}$ with $r\geq7.$ That is, $\hat{G}[U] = K_{1,r} \cup (d(u^*)-r-1)K_1$ for some $r \geqslant 7.$ Moreover $e(W)=0.$
\end{lemma}
\noindent \textbf{Proof.} Under the assumption that $\lambda > \frac{1+\sqrt{4m-7}}{2}$, we immediately obtain $\lambda^{2}-\lambda > m-2.$ Combining this with \eqref{310},~\eqref{311} and \eqref{312}, it follows that
\begin{align*}
    \left(m-2\right)x_{u^{*}} &< (\lambda^{2}-\lambda)x_{u^{*}}\\
 &= |U|x_{u^{*}} + \sum_{u\in U^{+}}(d_{U}(u)-1)x_{u}- \sum_{v\in U_{0}} x_{v}
+ \sum_{w\in W} d_{U}(w)x_{w}\\
&\le \bigl(|U| + 2e(U_{+}) - |U_{+}| + e(U,W)\bigr)x_{u^{*}}\\
&= \bigl(m - e(W) + e(U_{+}) - |U_{+}|\bigr)x_{u^{*}}.
\end{align*}
\noindent Consequently, we obtain
\begin{align}\label{W}
e(W) < e(U_{+}) - |U_{+}| + 2.
\end{align}

\noindent By Lemma \ref{lem24}, each component of $\hat G[U_{+}]$ is isomorphic
to a star $K_{1,r}$ with $r\geq1$. Now, if there is more than one star component in $\hat G[U_{+}]$, then $e(U_{+}) - |U_{+}| \le -2.$ This together with \eqref{W} yeilds
\begin{align*}
e(W) < e(U_{+}) - |U_{+}| + 2 \le -2 + 2 = 0,
\end{align*}
which is a contradiction.
This with Lemma \ref{C1} shows that $\hat G[U_{+}]$ consists of exactly one star $K_{1,r}$, where $r \ge 7$.\\
\indent It is clear that $e(U_{+}) - |U_{+}| = -1$.
In view of (\ref{W}), we get
\begin{align*}
e(W) < e(U_{+}) - |U_{+}| + 2= -1 + 2 = 1.
\end{align*}
This implies that $e(W)=0.$
\qed
\begin{lemma}\label{C3}
$d(w)\geq 2$ for each $w\in W.$
\end{lemma}
\noindent \textbf{Proof.} Suppose, to the contrary, that there exists a vertex $w\in W$ such that $d(w)=1$ and $wv\in E(G).$ Clearly, $v\neq u^*.$ Let $G^\prime=G-wv+wu^*.$ Since $x_{u^*}=\max\{x_u:u\in V(G)\},$ we have $x_{u^*}\geq x_w.$ It is clear that $G^\prime$ is an $H(4,3)-$free graph with $m$ edges. By Lemma \ref{lem25}, we obtain $\lambda(G^\prime) > \lambda(G),$ which yields a contradiction.
\qed
\indent The following useful observation will be required in the sequel.\\

\noindent\textbf{Observation 1.} $\displaystyle\sum_{v\in U_0} \frac{x_v}{x_{u^*}} < 1.$ That is, $\displaystyle\sum_{v\in U_0} x_v < x_{u^*}.$ \\
\noindent \textbf{Proof.} Recall from \eqref{555} that
 \begin{align*}
0\leq e(W) < e(U_+)-|U_+|-\sum_{v\in U_0}\frac{x_v}{x_{u^*}}+2.
\end{align*}
In view of Lemma \ref{C2}, this inequality yields
\begin{align*}
\displaystyle\sum_{v\in U_0}\frac{x_v}{x_{u^*}} < e(U_+)-|U_+|+2=-1+2=1.
\end{align*}
In particular, \(
\sum_{v \in U_0} x_v < x_{u^*}.
\)
\qed
\begin{lemma}\label{C4}
$ W=\emptyset$
\end{lemma}
\noindent \textbf{Proof.} To the contrary, assume that $W \neq \emptyset$. By Lemma \ref{C1}, we have $\hat{G}[U] \cong K_{1,r} \cup (d(u^*)-r-1)K_1$ for some $r\geq 7$. Moreover, $e(W)=0$ and  $\sum_{v\in U_0}x_v < x_{u^*}.$ Let  $V(U_+)=\{v_0,v_1,\dots,v_r\}$ denote the vertex set of the star $K_{1,r},$ where $v_0$ is the center. Since $r\geq 7$, we have the following cases to consider.\\
\noindent \textbf{Case 1.} $N_{W}(v_0) \neq \emptyset$. In this case, we may assume $N_{W}(v_0)=\{w_1,w_2,\ldots,w_s\}.$ Since $\hat{G}$ is $H(4,3)$-free and $e(W)=0$, we have
$N(w_i)\subseteq \{v_0\}\cup U_0 \quad \text{for } i\in\{1,2,\ldots,s\}.$\\
By \eqref{111}, we get\\
\begin{align*}
x_{v_0}\le \lambda x_{w_i}\le x_{v_0}+\sum_{v\in U_0} x_v < x_{v_0}+x_{u^*}
\quad \text{for } i\in\{1,2,\ldots,s\}.
\end{align*}
Equivalently,
\begin{align*}
\frac{1}{\lambda}x_{v_0}\le x_{w_i}<\frac{1}{\lambda}(x_{v_0}+x_{u^*})\le \frac{2}{\lambda}x_{u^*}\quad \text{for } i\in\{1,2,\ldots,s\}.
\end{align*}
\noindent Again, by \eqref{111}, we have

\begin{align*}
\lambda x_{u^*} = x_{v_0}+\sum_{i=1}^{r} x_{v_i}+\sum_{v\in U_0} x_v
< x_{v_0}+\sum_{i=1}^{r} x_{v_i}+x_{u^*}
\end{align*}
and
\begin{align*}
\lambda x_{v_0} = x_{u^*}+\sum_{i=1}^{r} x_{v_i}+\sum_{i=1}^{s} x_{w_i}.
\end{align*}
Consequently,
 \begin{align*}
\lambda(x_{u^*}-x_{v_0}) < x_{v_0}-\sum_{i=1}^{s} x_{w_i}.
\end{align*}
Note that $x_{v_0}\le x_{u^*}$. So it follows that $\sum_{i=1}^{s} x_{w_i}<x_{v_0}.$
This implies that $x_{w_i}<x_{v_0}$ for $i\in\{1,2,\ldots,s\}$. By Lemma \ref{C3}, we have
$d(w_i)\ge 2 \quad \text{for } i\in\{1,2,\ldots,s\}.$ It is easy to see that there exists a vertex, say $u_i\in N(w_i)\cap U_0$ for $i\in\{1,2,\ldots,s\}.$ Construct a new graph
$G_1=\hat{G}-w_1u_1+v_0u_1.$ It is clear that $G_1\in \mathcal{G}(m,H(4,3))$. By Lemma \ref{lem25}, we get $\lambda(G_1)>\lambda(\hat{G})$, a contradiction.

\noindent \textbf{Case 2.} $N_{W}(v_{0})=\emptyset$. In this case, we get $N(w)\subseteq \{v_{1},v_{2},\dots,v_{r}\}\cup U_{0}$ for every $ w\in W.$ We consider the following two subcases.\\
\noindent \textbf{Subcase 2.1.} Let $U_{0}\neq \emptyset$ and write $U_{0}=\{u_{1},u_{2},\ldots,u_{|U_{0}|}\}$. We claim that
$N(w)\cap U_{0}=\emptyset$. Indeed, to the contrary, suppose that $wu_i\in E(\hat{G})$ for some
$w\in W$ and some $1\le i\le |U_{0}|$. Define
\[
G_{2}=\hat{G}-wu_i+v_{0}u_i.
\]
Clearly, $G_{2}\in \mathcal{G}(m,H(4,3))$. Since $w\in W$ and
$\sum_{v\in U_{0}} x_v < x_{u^*}$, from \eqref{111}, it follows that
\begin{align*}
\lambda x_{w}\le \sum_{i=1}^{r} x_{v_{i}}+\sum_{v\in U_{0}} x_{v}
< \sum_{i=1}^{r} x_{v_{i}}+x_{u^*}
= \lambda x_{v_{0}}.
\end{align*}
This yields $x_w < x_{v_0}.$ By Lemma \ref{lem25}, it follows that $\lambda(G_2) > \lambda(\hat{G}),$ a contradiction.\\
Thus, $N(w)\cap U_{0}=\emptyset$ for every $w\in W.$ Hence, $N(w)\subseteq \{v_{1},v_{2},\dots,v_{r}\}$ for every $w\in W.$ For convenience, we let $b=\sum_{w\in W} d(w).$
First, let $b$ be even. Then consider the graph $G_{3}$ obtained from $\hat{G}$ by adding $\frac{b}{2}$ isolated vertices $\{p_{1},p_{2},\ldots,p_{b/2}\},$ and construct a new graph
\begin{align*}
G_{4}= G_{3}- \sum_{w\in W}\sum_{v\in N(w)} wv
+ \sum_{i=1}^{b/2} p_i u^*+ \sum_{i=1}^{b/2} p_i v_{0}.
\end{align*}

Clearly, $G_{4}\in \mathcal{G}(m,H(4,3))$. Let $\mathbf{y}$ (resp.\ $\mathbf{z}$) denote
the Perron vector of $G_{3}$ (resp.\ $G_{4}$), where $y_v$ (resp.\ $z_v$) corresponds
to the vertex $v\in V(G_{3})$ (resp.\ $v\in V(G_{4})$). Evidently,
$y_{p_i}=0$ for all $i\in\{1,2,\ldots,b/2\}$, and $y_v=x_v$ for every other vertex in
$V(G_{3})$. Moreover, it is straightforward to verify that
$y_w < y_{v_0} < y_{u^*}$ for each $w\in W$. Furthermore, we have
\[
z_{v_1}=z_{v_2}=\dots=z_{v_r}
= z_{p_1}=z_{p_2}=\dots=z_{p_{b/2}},
\]
and $z_w=0$ for every $w\in W$. Then we obtain
\begin{align*}
    (\lambda(G_{4})-\lambda(\hat{G}))\,\mathbf{y}^{T}\mathbf{z}
  &= \mathbf{y}^{T}\lambda(G_{4})\mathbf{z}
  - \lambda(\hat{G})\,\mathbf{y}^{T}\mathbf{z} \\
  &= \mathbf{y}^{T}A(G_{4})\mathbf{z}
  - \mathbf{y}^{T}A(\hat{G})\mathbf{z} \\
  &= \mathbf{y}^{T}\bigl(A(G_{4})-A(\hat{G})\bigr)\mathbf{z}\\
  &= \sum_{i=1}^{b/2} (y_{u^*} z_{p_i} + z_{u^*} y_{p_i})
   +\sum_{i=1}^{b/2} (y_{v_0} z_{p_i} + z_{v_0} y_{p_i})-\sum_{w\in W}\sum_{v\in N(w)}(y_wz_v+z_wy_v)\\
  &= \sum_{i=1}^{b/2} y_{u^*} z_{p_i}
   + \sum_{i=1}^{b/2} y_{v_0} z_{p_i}
   - \sum_{w\in W} \sum_{v\in N(w)} y_w z_v \\
  &> \frac{b}{2} y_{v_0} z_{v_1}
   + \frac{b}{2} y_{v_0} z_{v_1}
   - \sum_{w\in W} y_{v_0} d(w) z_{v_1} \\
   &= \left( \frac{b}{2} + \frac{b}{2} - \sum_{w\in W} d(w) \right)
   y_{v_0} z_{v_1} \\
  &= 0 .
\end{align*}
Hence, $\lambda(G_{4}) > \lambda(\hat{G})$, a contradiction.\\
\indent Now, let $b$ be odd. Recall that $U_{0}\neq \emptyset $ and $u_{1}\in U_{0}$.
Let $G_{5}$ be the graph obtained from $\hat{G}$ by adding $\frac{b-1}{2}$ isolated vertices $\{p_{1},p_{2},\ldots,p_{(b-1)/2}\},$ and construct a new graph
$$G_{6}= G_{5}- \sum_{w\in W}\sum_{v\in N(w)} wv
+ \sum_{i=1}^{(b-1)/2} p_i u^\ast+ \sum_{i=1}^{(b-1)/2} p_i v_{0}
+ v_{0}u_{1}.$$
Clearly, $G_{6}$ is also a graph in $\mathcal{G}(m,H(4,3))$. Let $\bar{\mathbf{y}}$ (resp.\ $\bar{\mathbf{z}}$) denote the Perron vector of $G_{5}$ (resp.\ $G_{6}$), where $\bar{y}_v$ (resp.\ $\bar{z}_v$) corresponds to the vertex $v\in V(G_{5})$ (resp.\ $v\in V(G_{6})$). Obviously, $\bar{y}_{p_i}=0$ for $i\in\{1,2,\ldots,(b-1)/2\}$ and $\bar{y}_v=x_v$ for other vertices in $V(G_{5})$. Moreover, it is easy note that $\bar{y}_w < \bar{y}_{v_0} < \bar{y}_{u^\ast} $ for every $w\in W.$ Furthermore, we have \[\bar{z}_{v_1}=\bar{z}_{v_2}=\cdots=\bar{z}_{v_r}
= \bar{z}_{p_1}=\bar{z}_{p_2}=\cdots=\bar{z}_{p_{(b-1)/2}}
= \bar{z}_{u_1},\] and $\bar{z}_w=0 $ for every $w\in W.$ Then we obtain
\begin{align*}
(\lambda(G_{6})-\lambda(\hat{G}))\,\bar{\mathbf{y}}^{T}\bar{\mathbf{z}}
&= \bar{\mathbf{y}}^{T}\lambda(G_{6})\bar{\mathbf{z}}
  - \lambda(\hat{G})\,\bar{\mathbf{y}}^{T}\bar{\mathbf{z}} \\
&= \bar{\mathbf{y}}^{T}A(G_{6})\bar{\mathbf{z}}
  - \bar{\mathbf{y}}^{T}A(\hat{G})\bar{\mathbf{z}} \\
&= \bar{\mathbf{y}}^{T}\bigl(A(G_{6})-A(\hat{G})\bigr)\bar{\mathbf{z}} \\
&= \sum_{i=1}^{(b-1)/2}(\bar{y}_{u^*}\bar{z}_{p_i}+\bar{z}_{u^*}\bar{y}_{p_i})
 + \sum_{i=1}^{(b-1)/2}(\bar{y}_{v_0}\bar{z}_{p_i}+\bar{z}_{v_0}\bar{y}_{p_i}) \\
&\quad + (\bar{y}_{v_0}\bar{z}_{u_1}+\bar{z}_{v_0}\bar{y}_{u_1})
 - \sum_{w\in W}\sum_{v\in N(w)}(\bar{y}_w\bar{z}_v+\bar{z}_w\bar{y}_v) \\
&= \sum_{i=1}^{(b-1)/2}\bar{y}_{u^*}\bar{z}_{p_i}
 + \sum_{i=1}^{(b-1)/2}\bar{y}_{v_0}\bar{z}_{p_i}
 + \bar{y}_{v_0}\bar{z}_{u_1}
 + \bar{z}_{v_0}\bar{y}_{u_1}
 - \sum_{w\in W}\sum_{v\in N(w)}\bar{y}_w\bar{z}_v \\
&> \frac{b-1}{2}\bar{y}_{v_0}\bar{z}_{v_1}
 + \frac{b-1}{2}\bar{y}_{v_0}\bar{z}_{v_1}
 + \bar{y}_{v_0}\bar{z}_{v_1}
 - \sum_{w\in W}\bar{y}_{v_0}d(w)\bar{z}_{v_1} \\
&= \left( \frac{b-1}{2}+\frac{b-1}{2}+1-\sum_{w\in W}d(w) \right)
   \bar{y}_{v_0}\bar{z}_{v_1} \\
&= 0 .
\end{align*}
Hence, $\lambda(G_{6})>\lambda(\hat{G})$, a contradiction.\\
\noindent \textbf{Subcase 2.2.} Suppose that $U_{0}=\emptyset$. Since $N_{W}(v_{0})=\emptyset$, from
\eqref{111}, it follows that
\[
\lambda x_{u^*}=x_{v_{0}}+\sum_{i=1}^{r} x_i
\quad \text{and} \quad
\lambda x_{v_{0}}=x_{u^*}+\sum_{i=1}^{r} x_i .
\]
Subtracting these two equalities yields
\[
\lambda (x_{u^*}-x_{v_{0}})=x_{v_{0}}-x_{u^*},
\]
or equivalently,
\[
(\lambda-1)(x_{u^*}-x_{v_{0}})=0.
\]
Since $\lambda>8$, we conclude that $x_{u^*}=x_{v_{0}}$. As $W\neq \emptyset$,
let $w\in W$. Then

\begin{align*}
\lambda x_{w}= \sum_{i\in N(w)} x_{v_{i}}\le  \sum_{i=1}^{r} x_{v_{i}}
= \lambda x_{u^*}-x_{v_{0}}= (\lambda-1)x_{u^*}.
\end{align*}

\noindent This gives  $x_{w}\le \left(1-\frac{1}{\lambda}\right)x_{u^*}.$ Using Lemma \ref{lem26},  we obtain
\begin{align*}
e(W)< e(U) - |U_+| + 2- \frac{1}{\lambda}d_{U}(w)= 1 - \frac{1}{\lambda} d_{U}(w).
\end{align*}
If $d_{U}(w)\ge \lambda$, then $e(W) < 1 - \frac{1}{\lambda} d_{U}(w) \leq 0,$ a contradiction. Therefore, we must have $2 \le d_{U}(w) < \lambda$ for each $w\in W.$

\noindent \textbf{Subcase 2.2.1.} There exists a vertex $w\in W$ such that $2 \le d_{U}(w) < \frac{\lambda}{2}.$\\
\indent Since $\lambda x_{w} = \sum_{i\in N(w)} x_{v_{i}}\le \sum_{i\in N(w)} x_{u^*}= d_{U}(w)\, x_{u^*}< \frac{\lambda}{2} x_{u^*},$ we obtain $x_{w} < \frac{1}{2} x_{u^*}= \left(1-\frac{1}{2}\right)x_{u^*}.$ By Lemma \ref{lem27}, we have
\begin{align*}
e(W)< \bigl(e(U) - |U_+|\bigr)+ 2 - \frac{1}{2} d_{U}(w)= 1 - \frac{1}{2} d_{U}(w)\le 0,
\end{align*}
which is a contradiction.

\noindent \textbf{Subcase 2.2.2.} Let  $\frac{\lambda}{2} \le d_{U}(w) < \lambda,$ for every vertex $w$ in $W.$ For convenience, let $d_{U}(w)=s$. If $|W|=1$, then the vertex set of $\hat{G}$ has an equitable partition $\pi_{1} : V = V_{1} \cup V_{2} \cup V_{3} \cup V_{4}$ (see Fig. 1), and the corrosponding quotient matrix with respect to $\pi_{1}$ is

\[
A_{\pi_{1}}=
\begin{pmatrix}
1 & s & \frac{ m-3s-1}{2} & 0 \\
2 & 0 & 0 & 1\\
2 & 0 & 0 & 0\\
0 & s & 0 & 0
\end{pmatrix},
\]\\
where $2(r-s)=m-3s-1.$ By Lemma \ref{lem28},  $\lambda$ is the largest real root of the equation
$f(x)=0,$ where $f(x)= \det(xI_{4}-A_{\pi_{1}})= x^{4}-x^{3}+(1-m)x^{2}+s x-3s^{2}+ms-s.$ A direct calculation shows that $f(x)>0 $ for any $x> \frac{1+\sqrt{4m-7}}{2},$ which implies $\lambda \leq \frac{1+\sqrt{4m-7}}{2},$ a contradiction. This shows that $|W|\geq 2.$
Note that \[\lambda x_{u^*}=x_{v_{0}}+\sum_{i=1}^{r} x_{v_{i}}
\qquad\text{and}\qquad \lambda x_{v_{0}}=x_{u^*}+\sum_{i=1}^{r} x_{v_{i}}.\]
On subtracting the two equalities and noting that $\lambda>8$, we obtain $x_{u^*}=x_{v_{0}}$ and $\sum_{i=1}^{r} x_{v_{i}}=(\lambda-1)x_{u^*}.$ Moreover,
\begin{align*}
   \lambda^{2}x_{u^*}
&= \lambda x_{v_{0}} + \sum_{i=1}^{r} \lambda x_{v_{i}}
= x_{u^*} + \sum_{i=1}^{r} x_{v_{i}}+ \sum_{i=1}^{r}\left(
x_{u^*}+x_{v_{0}}+\sum_{w\in N(v_{i})\cap W} x_{w}\right)\\
&= (2r+1)x_{u^*}+ \sum_{i=1}^{r} x_{v_{i}}+ \sum_{w\in W} d_{U}(w)x_{w}.
\end{align*}
Observe that $\lambda x_{w}= \sum_{v\in N(w)} x_{v}\le \sum_{i=1}^{r} x_{v_{i}}, ~\text{that is} ~ x_{w}\le \frac{1}{\lambda}\sum_{i=1}^{r} x_{v_{i}}.$ Then from above it follows that
\begin{align*}
  \lambda^{2}x_{u^*}& \le (2r+1)x_{u^*}+ \left(1+\frac{1}{\lambda}\sum_{w\in W} d_{U}(w)\right)\sum_{i=1}^{r} x_{v_{i}}\\
  &= (2r+1)x_{u^*}+ \left(1+\frac{m-2r-1}{\lambda}\right)(\lambda-1)x_{u^*}\\
  &= m x_{u^*} + (\lambda-1)x_{u^*}- \frac{m-2r-1}{\lambda}x_{u^*}\\
  &\le m x_{u^*} + (\lambda-1)x_{u^*} - x_{u^*}.
\end{align*}
This shows that $\lambda^2 -\lambda-m+2 \leq 0 ,$ which gives $\lambda(G)\leq \frac{1+ \sqrt{4m-7}}{2},$ a contradiction.
\qed
\begin{figure}[ht]
\centering
\includegraphics[width=140mm]{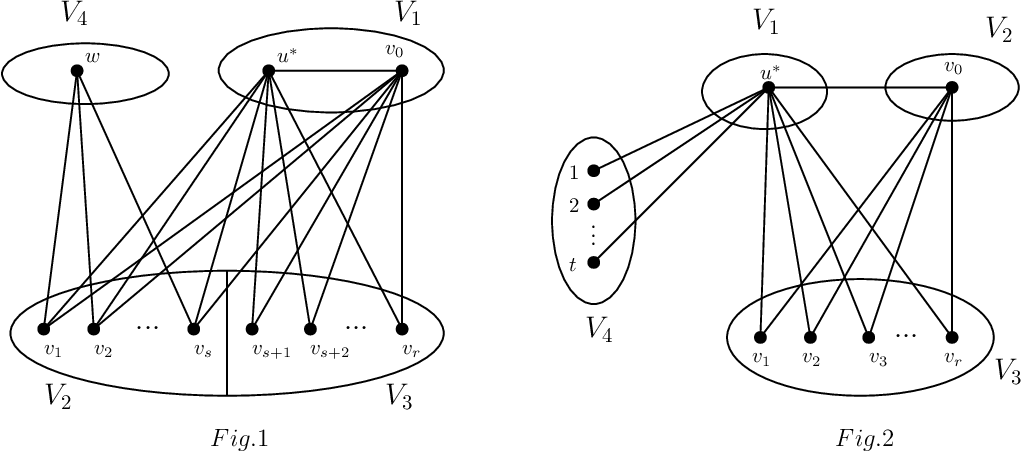}

\label{Bfig}
\end{figure}

\noindent \textbf{Proof of Theorem \ref{MT}} Let $|U_{0}| = t$. Observe that $\hat{G} \not\cong K_{1} \vee \bigl(K_{1,\frac{m-3}{2}} \cup 2K_{1}\bigr)$, where $m$ is odd. Consequently, $t$ is even and satisfies $t \ge 4$. Since $r \ge 7$, the vertex set of $\hat{G}$ admits an equitable partition \(\pi_{2} : V = V_{1} \cup V_{2} \cup V_{3} \cup V_{4}
\) (see Fig.~2), and the corresponding quotient matrix with respect to $\pi_{2}$ is

\[
A_{\pi_{2}}=
\begin{pmatrix}
0 & 1 & \frac{ m-t-1}{2} & t \\
1 & 0 & \frac{ m-t-1}{2} & 0\\
1 & 1 & 0 & 0\\
1 & 0 & 0 & 0
\end{pmatrix}.
\]\\
where, $2r=m-t-1.$ By Lemma \ref{lem28},  $\lambda$ is the largest root of the equation $f_{t}(x)=0,$ where $f_{t}(x)= \det(xI_{4}-A_{\pi_{2}})= x^{4}-mx^{2}-(m-t-1)x+\frac{t(m-t-1)}{2}.$ Since $f_t(x) - f_4(x)= (t-4)x + \frac{t(m-t-1)-4(m-5)}{2} >0$ for $x>0$ and $t\geq6$, it follows that $ t = |U_0| = 4$ for the extremal graph $\hat{G} $. Consequently, $\lambda$ is the largest root of $f_4(x) = x^4 - m x^2 - (m-5)x + 2(m-5).$\\
\indent Now, let $q(x) = x^2 - x - (m-2)$. Then $q\left(\frac{1 + \sqrt{4m-7}}{2}\right) = 0.$ A direct computation shows that $f_4(x) = (x^2 + x - 1) q(x) + 2x + (m-8).$ Clearly, $f_4(x) > 0 $ for $m\geq 58$ and $x > \frac{1 + \sqrt{4m-7}}{2}.$ This implies that $\lambda\leq\frac{1 + \sqrt{4m-7}}{2},$ a contradiction. This completes the proof.\qed

\section{Concluding remarks.}

In this paper, we established a sharp upper bound on the spectral radius of graphs over $ \mathcal{G}(m,H(4,3))\setminus\left\{K_2 \vee \frac{m-1}{2}K_1\right\} $ for odd $m\geq58$ and identified the unique extremal graph. This makes further progress towards the Brualdi-Hoffman-Tur$\acute{a}$n-type problem. Recently, Zhang, Wang \cite{L23} showed that $\lambda(G) \le \tilde{\lambda}(m)$ holds for all $H(4,3)-$free graphs of even size $m\geq51,$ where $\tilde{\lambda}(m)$ is the largest root of $x^4 - m x^2 - (m-2)x + \frac{m}{2}-1= 0$ and the equality holds if and only if $G\cong S^-_{\frac{m+4}{2},2},$ where $S^-_{\frac{m+4}{2},2}$ is the graph obtained from $K_2 \vee \frac{m}{2}K_1$ by deleting an edge incident to a vertex of degree two. So, it is natural to explore the upper bound on the spectral radius of graphs among $\mathcal{G}(m,H(4,3))\setminus S^-_{\frac{m+4}{2},2}$ for even size. Li, Lu and Peng \cite{LLP} focused the attention on  the spectral extremal problems for $F_2-$free graphs with given size as follows.
\begin{theorem}\label{333}\cite{LLP}
    If $m\ge 8$ and $G$ is an $F_2$-free graph with $m$ edges, then $\lambda(G)\leq \frac{1+\sqrt{4m-3}}{2},$ with equality if and only if $G\cong K_2\vee \frac{m-1}{2}K_1$.
\end{theorem}
It is important to note that in Theorem \ref{333}, the upper bound is sharp only for odd $m.$ In this connection, Chen, Li,  Zhang, Zhang \cite{L21} obtained a sharp upper bound on the spectral radius of graphs in $\mathcal{G}(m,F_2)$ for even $m\geq 16$ and identified the unique extremal graph. Furthermore, they established a sharp upper bound on the spectral radius of graphs over
 $ \mathcal{G}(m,F_2)\setminus\left\{K_2 \vee \frac{m-1}{2}K_1\right\} $ for odd $m \geq 17$  and identified the unique extremal graph. It is natural to take this investigation forward by considering the case of even size.

\noindent{\bf{Acknowledgement.}} The research of Abdul Basit Wani is supported by JRF, financial assistance by Department of Science and Technology (DST), Ministry of Science and Technology, Government of India, New Delhi-110016, No. DST/INSPIRE Fellowship/2024/IF240141.

\noindent{\bf Conflict of interest.} The authors declare that they have no conflict of interest.

\noindent{\bf Data Availibility.} Data sharing is not applicable to this article as no datasets were generated or analyzed
during the current study.\\

\end{document}